\documentclass[10pt]{article}
\usepackage[T1]{fontenc}
\usepackage[cp1250]{inputenc}
\usepackage{lmodern}

\usepackage[english]{babel}
\usepackage{latexsym}
\usepackage{amsmath}
\usepackage{indentfirst}
\usepackage{amsthm}
\usepackage{amssymb}
\usepackage{amsfonts}
\usepackage{stackrel}
\usepackage{dsfont}
\usepackage{tikz}
\usepackage{slashbox}
\usepackage[mathscr]{eucal}
\usepackage{authblk}
\usepackage{hyperref}
\usepackage{mathabx}
\usepackage{bookmark}

\newtheorem{thm}{Theorem}
\newtheorem{lemma}[thm]{Lemma}

\newtheorem{cor}[thm]{Corollary}

\newcommand{\reals}{\mathbb{R}}
\newcommand{\naturals}{\mathbb{N}}
\newcommand{\integers}{\mathbb{Z}}

\newcommand{\eps}{\varepsilon}

\newcommand{\supp}{\text{supp}}

\newcommand{\Ffamily}{\mathcal{F}}
\newcommand{\algebra}{\mathcal{A}}

\newcommand{\Schwartz}{\mathcal{S}(\mathbb{R}_+)}
\newcommand{\Schwartzc}{\mathcal{S}_c(\mathbb{R}_+)}
\newcommand{\SchwartzR}{\mathcal{S}(\mathbb{R})}
\newcommand{\MM}{\mathfrak{M}}

\begin{document}
\title{From Schwartz space to Mellin transform}
\author{Mateusz Krukowski}
\affil{Institute of Mathematics, \L\'od\'z University of Technology, \\ W\'ol\-cza\'n\-ska 215, \
90-924 \ \L\'od\'z, \ Poland \\ \vspace{0.3cm} e-mail: mateusz.krukowski@p.lodz.pl}
\maketitle

\begin{abstract}
The primary motivation behind this paper is an attempt to provide a thorough explanation of how the Mellin transform arises naturally in a process akin to the construction of the celebrated Gelfand transform. We commence with a study of a class of Schwartz functions $\mathcal{S}(\mathbb{R}_+),$ where $\reals_+$ is the set of all positive real numbers. Various properties of this Fr\'echet space are established and what follows is an introduction of the Mellin convolution operator, which turns $\Schwartz$ into a commutative Fr\'echet algebra. We provide a simple proof of Mellin-Young convolution inequality and go on to prove that the structure space $\Delta(\mathcal{S}(\mathbb{R}_+),\star)$ (the space of nonzero, linear, continuous and multiplicative functionals $m:\mathcal{S}(\mathbb{R}_+)\longrightarrow \reals$) is homeomorphic to $\reals.$ Finally, we show that the Mellin transform arises in a process which bears a striking resemblance to the construction of the Gelfand transform. 
\end{abstract}

\vspace{0.2cm}

\section{Introduction}
\label{Chapter:Introduction}

Mellin transform is a less popular, yet still very well-known twin brother of the celebrated Fourier transform. It is named after the Finnish mathematician Hjalmar Mellin, who introduced it in a paper published in 1897. Although well over a century has passed, the Mellin transform still enjoys an undiminishing attention from the mathematical community. In order to support that claim, let us mention just a couple of examples from the past two decades, which by no means exhaust the vast literature on the subject:
\begin{itemize}
	\item In 2001, Matti Jutila reserched the Mellin transform of the square of Riemann's zeta function (see \cite{Jutila}).
	
	\item In 2002, Paul L. Butzer, Anatoly A. Kilbas and Juan J. Trujillo applied the Mellin transform to Hadamard-type fractional integrals (see \cite{ButzerKilbasTrujillo} and \cite{ButzerKilbasTrujillo2}).
	
	\item In 2006, Jason Twamley and G. J. Milburn analyzed the quantum Mellin transform (see \cite{TwamleyMilburn}).
	
	\item In 2008, Victor Guillemin and Zuoqin Wang applied the twisted Mellin transform to problems in toric geometry (see \cite{GuilleminWang}).

	\item In 2014, Salvatore Butera and Mario Di Paola used the Mellin transform to solve fractional differential equations (see \cite{ButeraDiPaola}).

	\item In 2015, Carlo Bardaro, Paul L. Butzer and Ilaria Mantellini studied the foundations of fractional calculus in the Mellin transform setting (see \cite{BardaroButzerMantellini}).

	\item In 2016, Carlo Bardaro, Paul L. Butzer, Ilaria Mantellini and Gerhard Schmeisser established a version of the Paley-Wiener theorem of Fourier analysis in the frame of Mellin transforms (see \cite{BardaroButzerMantelliniSchmeisser}).

	\item In 2016, Hiroaki Matsueda examined holographic renormalization by singular value decomposition and found a mathematical form of each SVD component by the inverse Mellin transform (see \cite{Matsueda}).

	\item In 2019, Alexander E. Patkowski offered two new Mellin transform evaluations for the Riemann zeta function in the critical strip (see \cite{Patkowski}).
\end{itemize}

Seeing that the topic of Mellin transform is still an active field of mathematical research, we have decided to contribute by investigating the parallels between the constructions of Gelfand and Mellin transform. Let us summarize the structure of the paper in order to facilitate the comprehension of the ``big picture'' before delving into technical subtleties. 

In Section \ref{section:Schwartzfunctions} we define the class of Schwartz functions on positive real numbers. This is the base setting in which the Mellin transform will ``be born''. We explain why the Schwartz class is a Fr\'echet space but not a Banach space (since it is not normable). The section concludes with the introduction of a dense subset of the Schwartz class, which is easier to work with than the whole class itself.

Section \ref{section:Mellinconvolution} focuses on the study of the Mellin transform. The climactic point of this section is a proof that the Schwartz class together with the Mellin transform as multiplication constitutes a Fr\'echet algebra. What follows is a brief digression on Mellin-Young convolution inequality. 

Section \ref{section:structurespaceandMellintransfrom} is devoted to the structure space of the Schwartz class. This is a space of all nonzero, linear and multiplicative functionals on the Schwartz class, which we equip with the weak* topology. We prove that the structure space is bijective, and later on in fact homeomorphic to $\reals$ with the standard topology -- this is reminiscent of the classical results in standard Gelfand theory. 

The final Section \ref{section:finalremarks} serves as an epilogue and capitalizes on previously proven results. We explain how the Mellin transform arises in a similar vein to the construction of the Gelfand transform. At the same time, we stay aware and keep track of the differences between the two transforms.  Lastly, the paper concludes with a bibliography.

\section{Schwartz functions on positive reals}
\label{section:Schwartzfunctions}

We commence our study with a definition of a vector space of Schwartz functions on positive real numbers $\reals_+$:
\begin{gather}
\Schwartz := \bigg\{f\in C^{\infty}(\reals_+)\ :\ \forall_{\substack{\alpha\in\integers\\ \beta \in \naturals_0}}\ \exists_{M_{\alpha,\beta} > 0}\ \sup_{x\in\reals_+}\ |x^{\alpha} f^{(\beta)}(x)| < M_{\alpha,\beta} \bigg\},
\label{definitionofSchwartz}
\end{gather}

\noindent
where $\integers$ and $\naturals_0$ stand for integer and natural numbers with $0$, respectively. Our definition is reminiscent of the classical Schwartz class (see Chapter 8.1 in \cite{Follandrealanalysis}):
\begin{gather}
\SchwartzR := \bigg\{f\in C^{\infty}(\reals)\ :\ \forall_{\alpha,\beta\in\naturals_0}\ \exists_{M_{\alpha,\beta} > 0}\ \sup_{x\in\reals}\ (1+|x|)^{\alpha} |f^{(\beta)}(x)| < M_{\alpha,\beta} \bigg\}.
\label{definitionofSchwartzR}
\end{gather}

\noindent
We usually say that the classical Schwartz functions decrease rapidly in $\pm$infinity. Since the functions in $\Schwartz$ are defined on $\reals_+$, which is ``less symmetric'' than $\reals$, we may say that they are rapidly decreasing both at $0$ and $+$infinity. Hence, we use $x^{\alpha},\ \alpha\in\integers$ in definition \eqref{definitionofSchwartz} rather than $(1+|x|)^{\alpha},\ \alpha\in\naturals_0$ as in definition \eqref{definitionofSchwartzR}. 

$\Schwartz$ can be turned into a locally convex vector space (see Theorem 1.37 in \cite{RudinFunctionalAnalysis}, p. 27) by a family of separating seminorms $p_{\alpha,\beta}:\Schwartz\longrightarrow [0,\infty)$ given by
$$p_{\alpha,\beta}(f) := \sup_{x\in\reals_+}\ \left|x^{\alpha} f^{(\beta)}(x)\right|.$$

\noindent
A family of sets of the form 
$$V_{p_{\alpha_1,\beta_1}}(\eps)\cap \ldots\cap V_{p_{\alpha_n,\beta_n}}(\eps)$$ 

\noindent
for some $(\alpha_k)_{k=1}^n\subset \integers,\ (\beta_k)_{k=1}^n\subset \naturals_0,\ \eps>0,$ where 
$$V_{p_{\alpha_k,\beta_k}}(\eps) := \bigg\{f\in\Schwartz\ :\ p_{\alpha_k,\beta_k}(f) < \eps\bigg\}$$

\noindent
forms a neighbourhood base at zero for the locally convex topology of $\Schwartz.$ Furthermore, since $(p_{\alpha,\beta})_{\alpha\in\integers,\beta\in\naturals_0}$ is a countable family, then the topology of $\Schwartz$ is metrizable (see Lemma 5.75 in \cite{AliprantisBorder}, p. 206) by a translation-invariant metric $d_{\Schwartz}:\Schwartz \times \Schwartz \longrightarrow [0,\infty)$ given by the formula
$$d_{\Schwartz}(f,g) := \sum_{\substack{\alpha\in\integers,\\ \beta\in\naturals_0}}\ \frac{2^{-(|\alpha|+|\beta|)}p_{\alpha,\beta}(f-g)}{1+p_{\alpha,\beta}(f-g)}.$$

\noindent
In particular this means that a sequence $(f_m)$ converges to $f$ in $\Schwartz$ if and only if 
$$\forall_{\substack{\alpha\in\integers,\\ \beta \in\naturals_0}}\ \lim_{m\rightarrow \infty}\ p_{\alpha,\beta}(f_m-f) = 0.$$ 

\noindent
This also means that a convergent sequence $(f_m)$ is topologically bounded (see Lemma 5.76 in \cite{AliprantisBorder}, p. 206), because $d_{\Schwartz}-$boundedness is obvious (recall that in general topological boundedness and $d_{\Schwartz}-$boundedness need not coincide -- see Chapter 1.29 in \cite{RudinFunctionalAnalysis}, p. 23).

The following result will be very useful in the sequel:

\begin{lemma}
Let $B \subset \Schwartz$ be a topologically bounded set. For every $\alpha\in\integers,\ \beta\in\naturals_0$ and $p\geqslant 1$ there exists a constant $M>0$ such that 
$$\forall_{f\in B}\ \int_{\reals_+}\ \left|x^{\alpha}f^{(\beta)}(x)\right|^p\ dx < 2M^p.$$
\label{Schwartzfunctionintegrable}
\end{lemma}
\begin{proof}
Since $B$ is topologically bounded set then (see Lemma 5.76 in \cite{AliprantisBorder}, p. 206 or Theorem 1.37(b) in \cite{RudinFunctionalAnalysis}, p. 28) there exists a constant $M>0$ such that
$$\forall_{f \in B}\ p_{\alpha,\beta}(f) + p_{\alpha+2,\beta}(f) < M.$$

\noindent
In particular, for every $f \in B$ we have
\begin{gather*}
\forall_{x\in (0,1]}\ x^{\alpha}\left|f^{(\beta)}(x)\right| < M \hspace{0.4cm}\text{and}\hspace{0.4cm} \forall_{x\in [1,\infty)}\ x^{\alpha+2}\left|f^{(\beta)}(x)\right| < M.
\end{gather*}

\noindent
In order to see that $M$ is the desired constant, we compute
$$\forall_{f\in B}\ \int_0^{\infty}\ \left|x^{\alpha}f^{(\beta)}(x)\right|^p\ dx < M^p + \int_1^{\infty}\ \frac{M^p}{x^{2p}}\ dx \leqslant 2M^p,$$

\noindent
which concludes the proof.
\end{proof}

Our next goal is to prove that $\Schwartz$ is a Fr\'echet space. Given what we said above, the only piece of information we are missing is whether $d_{\Schwartz}$ is \textit{complete} or not. In order to solve this puzzle we will need the following lemma:

\begin{lemma}
If the sequence $(f_n)\subset C^1(\reals_+)$ converges uniformly to $f$ and $(f_n')$ converges uniformly to $g$ then $f\in C^1(\reals_+)$ and $f' = g.$
\label{uniformconvergencewithderivative}
\end{lemma}
\begin{proof}
Since $(f_n')$ is uniformly convergent to $g,$ then $g\in C(\reals_+).$ Furthermore, for fixed numbers $x\in\reals_+$ and $h\in (-x,\infty)$ we have 
\begin{gather*}
\frac{f(x+h) - f(x)}{h} = \lim_{n\rightarrow \infty}\ \frac{f_n(x+h) - f_n(x)}{h} = \lim_{n\rightarrow \infty}\ \frac{\int_x^{x+h}\ f_n'(t)\ dt}{h} = \frac{\int_x^{x+h}\ g(t)\ dt}{h}.
\end{gather*} 

\noindent
Consequently, we get
\begin{gather*}
\lim_{h\rightarrow 0}\ \frac{f(x+h) - f(x)}{h} = \lim_{h\rightarrow 0}\ \frac{\int_x^{x+h}\ g(t)\ dt}{h} = \lim_{h\rightarrow 0}\ g(x+h) = g(x),
\end{gather*} 

\noindent
where we have applied d'Hospital's rule. This concludes the proof. 
\end{proof}

\begin{thm}
$(\Schwartz,d_{\Schwartz})$ is a complete metric space. 
\end{thm}
\begin{proof}
Suppose that $(f_n)$ is a $d_{\Schwartz}-$Cauchy sequence. This implies that for every $\alpha\in\integers$ and $\beta\in\naturals_0$ the sequence $(x\mapsto x^{\alpha}f_n^{(\beta)}(x))_{n\in\naturals}$ is a Cauchy sequence in the supremum norm (on $\reals_+$). Consequently, every such sequence is uniformly convergent to some $f_{\alpha,\beta} \in C^b(\reals_+).$ 

We put $f := f_{0,0},$ i.e., $f$ is a uniform limit of the sequence $(f_n).$ By Lemma \ref{uniformconvergencewithderivative} (and an elementary induction) we have $f\in C^{\infty}(\reals_+)$ and $f^{(\beta)} = f_{0,\beta}.$ Finally, the sequence $(x\mapsto x^{\alpha}f_n^{(\beta)}(x))_{n\in\naturals}$ is (by definition) uniformly convergent to $f_{\alpha,\beta},$ so
$$\forall_{x\in\reals_+}\ f_{\alpha,\beta}(x) = \lim_{n\rightarrow\infty}\ x^{\alpha}f_n^{(\beta)}(x) = x^{\alpha}f^{(\beta)}(x).$$

\noindent
This proves that $f\in \Schwartz$ and concludes the proof. 
\end{proof}

At this point we have shown that $\Schwartz$ is a Fr\'echet space. Encouraged by this fact we could hope that $\Schwartz$ is even a normed space but, unfortunately, this is not true and we will use the following result to prove it:

\begin{lemma}
There exists a function $\eta\in C^{\infty}(\reals)$ such that 
\begin{enumerate}
	\item $\text{supp}(\eta) = [-1,1],$
	\item $\eta(t) > 0$ for every $t\in (-1,1),$
	\item $\eta(0) = 1,$
	\item $\eta'(t) = 2(\eta(2t+1) - \eta(2t-1))$ for every $t\in\reals,$
	\item for $\beta \in \naturals$ and $t\in [-1,1]$ we have
	$$\eta^{(\beta)}(t) = 2^{\binom{\beta+1}{2}}\sum_{l=0}^{2^{\beta}}\ (-1)^{s(l)}\eta\left(2^{\beta}t + 2^{\beta} - 2l - 1\right),$$
	
	\noindent
	where $s(l)$ is the sum of the digits of $l$ when written in base $2$,
	\item for $\beta \in \naturals$ we have 
	\begin{gather}
	\eta^{(\beta)}(t_{\beta}) = 2^{\binom{\beta+1}{2}},
	\label{estimateonphibetasupremum}
	\end{gather}
	
	\noindent
	where $t_{\beta+1} := 2t_{\beta} + 1$ and $t_1 := -\frac{1}{2}.$
\end{enumerate}
\label{existenceofsuperfunction}
\end{lemma}
\begin{proof}
The existence of function $\eta$ satisfying points \textit{1. -- 4.} stems from Theorem 1 whereas point \textit{5.} stems from Theorem 4 in \cite{AriasdeReyna}.

As far as point \textit{6.} is concerned, we prove it by induction -- for $\beta = 1$ we have 
$$\eta'(t_1) = 2(\eta(2t_1 + 1) - \eta(2t_1 - 1)) = 2\eta(0) = 2.$$

\noindent
Next, assuming (\ref{estimateonphibetasupremum}) holds true for $\beta \in \naturals$ we observe that 
$$\eta^{(\beta+1)}(t_{\beta+1}) =  2^{\beta+1}(\eta^{(\beta)}(2t_{\beta+1}+1) - \eta^{(\beta)}(2t_{\beta+1}-1)) = 2^{\beta+1}\eta^{(\beta)}(t_{\beta}) 
= 2^{\beta+1} 2^{\binom{\beta+1}{2}} = 2^{\binom{\beta+2}{2}},$$

\noindent
which concludes the proof.
\end{proof}

\begin{thm}
$\Schwartz$ is not normable.
\label{Schwartzspacenotnormable}
\end{thm}
\begin{proof}
By Theorem 1.39 in \cite{RudinFunctionalAnalysis}, p. 30 (or Theorem 5.77 in \cite{AliprantisBorder}, p. 206) the normability of a topological vector space like $\Schwartz$ is equivalent to the existence of a convex, bounded, open (and nonempty) neighbourhood of the origin. For the sake of contradiction, we suppose that such a neighbourhood exists in $\Schwartz$. Then it must contain a base set of the form 
$$V := V_{p_{\alpha_1,\beta_1}}(\eps)\cap \ldots\cap V_{p_{\alpha_n,\beta_n}}(\eps)$$ 

\noindent
for some $(\alpha_k)_{k=1}^n\subset \integers,\ (\beta_k)_{k=1}^n\subset \naturals_0,\ \eps>0,$ where 
$$V_{p_{\alpha_k,\beta_k}}(\eps) := \bigg\{f\in\Schwartz\ :\ p_{\alpha_k,\beta_k}(f) < \eps\bigg\}.$$

\noindent
Without loss of generality, we assume that $\beta_1\leqslant \beta_2\leqslant \ldots\leqslant \beta_n.$ 

Our approach boils down to the construction of a sequence $(f_m)\subset \Schwartz$ such that $(f_m) \subset V$ but 
\begin{gather}
\lim_{m\rightarrow \infty}\ p_{0,\beta_n+2}(f_m) = \infty.
\label{limitwithbetanplus2}
\end{gather}

\noindent
First, we define a function $f:\reals_+ \longrightarrow \reals$ by the formula $f(x) := \eta(x-2),$ where $\eta$ is the function from Lemma \ref{existenceofsuperfunction}. In particular, by point \textit{5.} in Lemma \ref{existenceofsuperfunction} we have 
$$\forall_{x\in[1,3]}\ f^{(\beta)}(x) = 2^{\binom{\beta+1}{2}}\sum_{l=0}^{2^{\beta}}\ (-1)^{s(l)}f\left(2^{\beta}x + 2^{\beta} - 2l + 1\right).$$

\noindent
This leads to the estimate 
\begin{gather}
\forall_{x\in\reals_+}\ |f^{(\beta)}(x)| \leqslant 2^{\binom{\beta+1}{2}} \left(2^{\beta} + 1\right) \|f\|_{\infty} < 2^{\binom{\beta+1}{2} + \beta + 1} \|f\|_{\infty}.
\label{estimateonfbeta}
\end{gather}

Next, we define a sequence of functions $f_m: \reals_+ \longrightarrow \reals$ by the formula 
$$f_m(x) := \frac{\eps}{3^{\max\{|\alpha_1|,\ldots,|\alpha_n|\}}\cdot 2^{\binom{\beta_n+1}{2} + \beta_n + 1} \cdot \|f\|_{\infty}}\cdot m^{-1}\cdot f\left(m^{\frac{1}{b_n+1}}x\right).$$

\noindent
Applying the estimate (\ref{estimateonfbeta}) we have
$$\forall_{m\in\naturals}\ \sup_{x\in\reals_+}\ |f_m^{(\beta)}(x)| < \frac{\eps}{3^{\max\{|\alpha_1|,\ldots,|\alpha_n|\}}\cdot 2^{\binom{\beta_n+1}{2} + \beta_n + 1}}\cdot m^{\frac{\beta}{\beta_n+1}-1}\cdot 2^{\binom{\beta+1}{2} + \beta + 1}.$$

\noindent
Since $\text{supp}(f) = [1,3]$ and $x^{\alpha} \leqslant 3^{|\alpha|}$ for $x\in[1,3]$ and $\alpha\in\integers$ we have
$$\forall_{m\in\naturals}\ \sup_{x\in\reals_+}\ |x^{\alpha}f_m^{(\beta)}(x)| < \eps\cdot \frac{3^{|\alpha|}}{3^{\max\{|\alpha_1|,\ldots,|\alpha_n|\}}} \cdot \frac{2^{\binom{\beta+1}{2} + \beta + 1}}{2^{\binom{\beta_n+1}{2} + \beta_n + 1}}\cdot m^{\frac{\beta}{\beta_n+1}-1}.$$

\noindent
Consequently, if $\alpha \in \{\alpha_1,\ldots,\alpha_n\}$ and $\beta \in \{\beta_1,\ldots,\beta_n\}$ then 
$$\forall_{m\in\naturals}\ \sup_{x\in\reals_+}\ |x^{\alpha}f_m^{(\beta)}(x)| < \eps$$

\noindent
which means that $(f_m) \subset V.$

We will now prove (\ref{limitwithbetanplus2}) and then explain how it results in nonormability of $\Schwartz$. By Lemma \ref{existenceofsuperfunction} we have
$$\forall_{\beta\in\naturals}\ f^{(\beta)}(t_{\beta}+2) = \eta^{(\beta)}(t_{\beta}) \geqslant 2^{\binom{\beta+1}{2}}.$$

\noindent
This implies that 
\begin{equation*}
\begin{split}
\forall_{\beta,m\in\naturals}\ \sup_{x\in\reals_+}\ |f_m^{(\beta)}(x)| &\geqslant f_m^{(\beta)}\left(\frac{t_{\beta}+2}{m^{\frac{1}{b_n+1}}}\right)
\geqslant \frac{\eps}{3^{\max\{|\alpha_1|,\ldots,|\alpha_n|\}}\cdot 2^{\binom{\beta_n+1}{2} + \beta_n + 1}\cdot \|f\|_{\infty}}\cdot m^{\frac{\beta}{\beta_n+1}-1}f^{(\beta)}(t_{\beta}+2) \\
&\geqslant \frac{\eps}{3^{\max\{|\alpha_1|,\ldots,|\alpha_n|\}}}\cdot \frac{2^{\binom{\beta+1}{2}}}{2^{\binom{\beta_n+1}{2} + \beta_n + 1}}\cdot \frac{m^{\frac{\beta}{\beta_n+1}-1}}{\|f\|_{\infty}},
\end{split}
\end{equation*}

\noindent
which in particular means that 
$$\lim_{m\rightarrow\infty}\ \sup_{x\in\reals_+}\ |f_m^{\beta_n+2}(x)| = \infty.$$

We are finally ready to finish off the proof. We have just established the existence of a sequence of functions $(f_m)\subset \Schwartz$ such that $(f_m)\subset V$ and limit (\ref{limitwithbetanplus2}) holds. This means that $V$ is not bounded, since otherwise (by Theorem 5.76 in \cite{AliprantisBorder}, p. 206 or Theorem 1.37(b) in \cite{RudinFunctionalAnalysis}, p. 28) the seminorm $p_{0,\beta_n+2}$ would be bounded on the sequence $(f_m)$. Consequently, we arrive at a contradiction since $V$ was supposed to be a subset of a bounded set. This concludes the proof.  
\end{proof}

Our final objective in this section is to find a dense subset of $\Schwartz$, which is relatively easier to work with than $\Schwartz$ as a whole. To this end, we will employ the following lemma:

\begin{lemma}
There exists a sequence of functions $(\theta_n) \subset \Schwartz$ such that for every $n\in\naturals,\ n\geqslant 2$ we have
\begin{enumerate}
	\item $\text{supp}(\theta_n) = \left[\frac{1}{n}, n+1\right],$
	\item $\theta_n(x) = 1$ for $\left[\frac{2}{n}, n\right],$
	\item $|\theta_n^{(\beta)}(x)| \leqslant n^{\beta}2^{\frac{(\beta+1)\beta}{2}}$ for every $x\in\reals_+$ and $\beta \in \naturals.$
\end{enumerate}
\label{sequencethetan}
\end{lemma}
\begin{proof}
Let $\theta:[0,1]\longrightarrow[0,1]$ be the Fabius function (see \cite{Fabius}) and define the sequence of functions $(\theta_n)\subset\Schwartz$ by 
\begin{gather*}
\theta_n(x) := \left\{\begin{array}{cc}
\theta(nx - 1) & \text{for }\ x\in\left[\frac{1}{n}, \frac{2}{n}\right],\\
1 & \text{for }\ x\in\left[\frac{2}{n}, n\right],\\
\theta(n+1 - x) & \text{for }\ x\in[n,n+1],\\
0 & \text{for }\ x\not\in\left[\frac{1}{n}, n+1\right].
\end{array}\right.
\end{gather*}

\noindent
Since $\theta'(x) = 2\theta(2x)$ for $x\in\left[0,\frac{1}{2}\right],$ an easy induction shows that for every $\beta\in\naturals$ we have 
$$\forall_{x\in\left[0,\frac{1}{2}\right]}\ |\theta^{(\beta)}(x)| \leqslant 2^{\frac{(\beta+1)\beta}{2}}.$$

\noindent
Consequently, we have
$$\forall_{x\in\reals_+}\ |\theta_n^{(\beta)}(x)| \leqslant n^{\beta} 2^{\frac{(\beta+1)\beta}{2}}$$

\noindent
which concludes the construction.
\end{proof}

\begin{thm}
The set 
$$\mathcal{S}_c(\reals_+) := \bigg\{f \in \Schwartz\ :\ \text{supp}(f)\text{ is compact}\bigg\}$$

\noindent
is dense in $\Schwartz.$
\label{densesubsetinSchwartz}
\end{thm}
\begin{proof}
Fix $f\in\Schwartz.$ We prove that the sequence of functions $(\theta_n\cdot f)$ converges to $f$ in $\Schwartz,$ i.e.,
$$\forall_{\substack{\alpha\in\integers,\\ \beta\in\naturals_0}}\ \lim_{n\rightarrow\infty}\ p_{\alpha,\beta}(\theta_n\cdot f - f) = 0.$$

Firstly, we fix $\alpha\in\integers$ and $\beta\in\naturals_0$. Since $f\in\Schwartz$ then there exists a constant $M_1>0$ such that 
\begin{gather}
\forall_{n\in\naturals}\ \forall_{\substack{k=0,\ldots,\beta\\ x\in\left(0,\frac{2}{n}\right]}}\ |x^{\alpha}f^{(k)}(x)| \leqslant M_1 x^{\beta+1} \leqslant M_1 \left(\frac{2}{n}\right)^{\beta+1}.
\label{importantestimate1}
\end{gather}

\noindent
Similarly, we argue the existence of a constant $M_2>0$ such that
\begin{gather}
\forall_{n\in\naturals}\ \forall_{\substack{k=0,\ldots,\beta\\ x\in[n,\infty)}}\ |x^{\alpha}f^{(k)}(x)| \leqslant \frac{M_2}{x^{\beta+1}} \leqslant \frac{M_2}{n^{\beta+1}}.
\label{importantestimate2}
\end{gather}

Next, let us note that for every $\beta\in\naturals_0$ and every $x\in\left[\frac{1}{n},\frac{2}{n}\right]\cup[n,n+1]$ we have
\begin{gather*}
(\theta_n\cdot f)^{(\beta)}(x) = \sum_{k=0}^{\beta-1}\ \binom{\beta}{k} \theta_n^{(\beta-k)}(x)f^{(k)}(x) + f^{(\beta)}(x)
\end{gather*}

\noindent
so
\begin{equation}
\begin{split}
|x^{\alpha}(\theta_n\cdot f - f)^{(\beta)}(x)| &\leqslant \sum_{k=0}^{\beta-1}\ \binom{\beta}{k} |\theta_n^{(\beta-k)}(x)| |x^{\alpha}f^{(k)}(x)| \\
&\stackrel{\text{Lemma } \ref{sequencethetan}}{\leqslant} \sum_{k=0}^{\beta-1}\ \binom{\beta}{k} n^{\beta-k}2^{\frac{(\beta-k+1)(\beta-k)}{2}} |x^{\alpha}f^{(k)}(x)|\\
&\leqslant n^{\beta} 2^{\frac{(\beta+1)\beta}{2}} \sum_{k=0}^{\beta-1}\ \binom{\beta}{k} |x^{\alpha}f^{(k)}(x)|\\
&\stackrel{\eqref{importantestimate1},\ \eqref{importantestimate2}}{\leqslant} n^{\beta} 2^{\frac{(\beta+1)\beta}{2}} \left(M_1\left(\frac{2}{n}\right)^{\beta+1} + \frac{M_2}{n^{\beta+1}}\right) \sum_{k=0}^{\beta-1}\ \binom{\beta}{k} \\
&= \frac{2^{\frac{(\beta+1)\beta}{2}} (2^{\beta+1}M_1 + M_2) (2^{\beta} - 1)}{n}
\end{split}
\label{maincomputationforthetanf}
\end{equation}

\noindent
Finally, we have
\begin{equation*}
\begin{split}
\sup_{x\in\reals_+}\ |x^{\alpha}(\theta_n\cdot f - f)^{(\beta)}(x)| &\leqslant \sup_{x\in\left(0,\frac{1}{n}\right]}\ |x^{\alpha}f^{(\beta)}(x)| + \sup_{x\in\left[\frac{1}{n},\frac{2}{n}\right]\cup\left[n,n+1\right]}\ |x^{\alpha}(\theta_n\cdot f - f)^{(\beta)}(x)| + \sup_{x\in[n+1,\infty)}\ |x^{\alpha}f^{(\beta)}(x)| \\
&\stackrel{\eqref{importantestimate1}, \eqref{importantestimate2}, \eqref{maincomputationforthetanf}}{\leqslant} M_1 \left(\frac{2}{n}\right)^{\beta+1} + \frac{2^{\frac{(\beta+1)\beta}{2}} (2^{\beta+1}M_1 + M_2) (2^{\beta} - 1)}{n} + \frac{M_2}{n^{\beta+1}},
\end{split}
\end{equation*}

\noindent
which converges to zero as $n\rightarrow \infty.$ This concludes the proof. 
\end{proof}

\section{Mellin convolution}
\label{section:Mellinconvolution}

The previous section focused on proving topological properties of $\Schwartz$ such as the fact that it is a Fr\'echet space. In the current section we want to enhance our understanding of $\Schwartz$ by studying the algebraic aspect of this space. Thus, we begin by introducing the \textit{Mellin convolution} operator $\star : \Schwartz \times \Schwartz \longrightarrow \Schwartz$ with the formula:
\begin{gather}
\forall_{x\in \reals_+}\ f\star g(x) := \int_0^{\infty}\ y^{-1} f\left(xy^{-1}\right) g(y)\ dy.
\label{Mellinconvolution}
\end{gather}

\noindent
This formula is well-known in the literature concerning Mellin transform -- see Definition 4 in \cite{ButzerJansche} or Chapter 8.3 in \cite{BhattaDebnath}. The novelty introduce in this paper is the domain and codomain of the Mellin convolution as well as the study of the algebraic properties of this operator. 

Looking at the formula \eqref{Mellinconvolution}, the question that springs to mind is whether Mellin convolution is well-defined? In other words, does $f\star g$ define a Schwartz function? Prior to giving a positive answer to this question we show that given two functions $f,g \in \Schwartz,$ the value $f\star g(x)$ as defined by \eqref{Mellinconvolution} is finite for every value $x\in \reals_+$ and that $f\star g\in C^{\infty}(\reals_+).$

\begin{thm}
If $f,g\in \Schwartz,\ \alpha \in\integers$ and $\beta\in\naturals_0$ then 
\begin{gather}
\int_0^{\infty}\ y^{-(\alpha+1)} f^{(\beta)}\left(xy^{-1}\right) g(y)\ dy
\label{thisshouldbefinite}
\end{gather}

\noindent
is well-defined, i.e., finite for every $x\in \reals_+.$ Furthermore, the equality 
\begin{gather}
\frac{d}{dx}\int_0^{\infty}\ y^{-(\alpha+1)} f^{(\beta)}\left(xy^{-1}\right) g(y)\ dy = \int_0^{\infty}\ y^{-(\alpha+2)} f^{(\beta+1)}\left(xy^{-1}\right) g(y)\ dy
\label{derivativeunderintegral}
\end{gather}

\noindent
holds, which implies that $f\star g \in C^{\infty}(\reals_+).$
\label{convolutionisCinfty}
\end{thm}
\begin{proof}
To begin with, since $f\in\Schwartz$ then there exists a constant $M_1 > 0$ such that 
$$\forall_{z\in\reals_+}\ |f^{(\beta)}(z)| < M_1\cdot z^{-(\alpha+1)}.$$

\noindent
Consequently, for every $x\in\reals_+$ we have
\begin{gather*}
\left|\int_0^{\infty}\ y^{-(\alpha+1)} f^{(\beta)}\left(xy^{-1}\right) g(y)\ dy\right| < M_1 x^{-(\alpha+1)} \int_0^{\infty}\ |g(y)|\ dy < \infty, 
\end{gather*}

\noindent
since $g\in L^1(\reals_+)$ by Lemma \ref{Schwartzfunctionintegrable}. 

We proceed with the second part of the theorem and prove that formula \eqref{derivativeunderintegral} holds at an arbitrary point $x\in\reals_+.$ We begin by choosing $\eps>0$ such that $[x-\eps,x+\eps] \subset \reals_+$ and investigating the limit
$$\lim_{n\rightarrow \infty}\ \int_0^{\infty}\ y^{-(\alpha+1)}\cdot \frac{f^{(\beta)}\left(\frac{x+h_n}{y}\right) - f^{(\beta)}\left(\frac{x}{y}\right)}{h_n}\cdot g(y)\ dy,$$

\noindent
where $(h_n)\subset (-\eps,\eps)$ is an arbitrary sequence convergent to zero. By the mean value theorem we have 
$$\forall_{\substack{n\in\naturals\\ y\in\reals_+}}\ \left|\frac{f^{(\beta)}\left(\frac{x+h_n}{y}\right) - f^{(\beta)}\left(\frac{x}{y}\right)}{\frac{h_n}{y}}\right| \leqslant \sup_{z\in(-\eps,\eps)}\ \left|f^{(\beta+1)}\left(\frac{x+z}{y}\right)\right| < M_2$$

\noindent
for some constant $M_2 > 0.$ Consequently, we obtain
$$\forall_{\substack{n\in\naturals\\ y\in\reals_+}}\ \bigg|\frac{f^{(\beta)}\left(\frac{x+h_n}{y}\right) - f^{(\beta)}\left(\frac{x}{y}\right)}{h_n}\bigg| \leqslant M_2 y^{-1}.$$

\noindent
and since $y\mapsto y^{-(\alpha+2)}g(y)$ is integrable over $\reals_+$ (by Lemma \ref{Schwartzfunctionintegrable}) then by the dominated convergence theorem we have 
$$\lim_{n\rightarrow \infty}\ \int_0^{\infty}\ y^{-(\alpha+1)}\cdot \frac{f^{(\beta)}\left(\frac{x+h_n}{y}\right) - f^{(\beta)}\left(\frac{x}{y}\right)}{h_n}\cdot g(y)\ dy = \int_0^{\infty}\ y^{-(\alpha+2)} f^{(\beta+1)}\left(xy^{-1}\right) g(y)\ dy.$$

\noindent
Since the sequence $(h_n)$ was chosen arbitrarily, we conclude that \eqref{derivativeunderintegral} holds true.
\end{proof}

\begin{cor}
If $f,g\in\Schwartz$ then $f\star g\in \Schwartz.$
\end{cor}
\begin{proof}
By Theorem \ref{convolutionisCinfty} we already know that $f\star g\in C^{\infty}(\reals_+)$ and 
$$\forall_{\beta\in\naturals_0}\ \frac{d^{\beta}}{dx^{\beta}} f\star g(x) = \int_0^{\infty}\ y^{-(\beta+1)} f^{(\beta)}\left(xy^{-1}\right) g(y)\ dy.$$

\noindent
Furthermore, since $f\in\Schwartz$ then for every $\alpha \in \integers, \beta \in \naturals_0$ there exists a constant $M > 0$ such that
$$\forall_{z\in\reals_+}\ |f^{(\beta)}(z)| < M z^{-\alpha},$$

\noindent
Hence we have
\begin{gather*}
\sup_{x\in\reals_+}\ |x^{\alpha} (f\star g)^{(\beta)}(x)| = \sup_{x\in\reals_+}\ \bigg|x^{\alpha} \int_0^{\infty}\ y^{-(\beta+1)} f^{(\beta)}\left(xy^{-1}\right) g(y)\ dy\bigg| \leqslant M \int_0^{\infty}\ y^{\alpha-(\beta+1)} |g(y)|\ dy < \infty,
\end{gather*}

\noindent
where the last inequality follows from the fact that $y\mapsto y^{\alpha-(\beta+1)} g(y)$ is integrable over $\reals_+$ (by Lemma \ref{Schwartzfunctionintegrable}). This concludes the proof. 
\end{proof}

At this point we might get a feeling that the space $\Schwartz$ is a ``natural habitat'' for the Mellin convolution to live in. However, this is by far not the end of the story. Recall that a Fr\'echet space $\Ffamily,$ together with multiplication $\odot:\Ffamily\times \Ffamily \longrightarrow \Ffamily$ is called a Fr\'echet algebra if 
\begin{itemize}
	\item for every $c \in \reals$ and $f,g\in \Ffamily:$
	$$c\cdot(f\odot g) = (c\cdot f)\odot g = f\odot(c\cdot g),$$
	
	\item multiplication is associative, i.e. for every $f,g,h\in\Ffamily:$
	$$(f\odot g)\odot h = f\odot (g\odot h),$$
	
	\item multiplication is distributive, i.e. for every $f,g,h\in\Ffamily:$
	\begin{equation*}
	\begin{split}
	f\odot(g+h) = f\odot g + f\odot h,\\
	(f+g)\odot h = f\odot h + g\odot h,
	\end{split}
	\end{equation*}
	
	\item multiplication is continuous. 
\end{itemize}

\noindent
In addition, if $f\odot g = g\odot f$ for every $f,g \in\Ffamily$ then $(\Ffamily,\odot)$ is called a commutative Fr\'echet algebra.

\begin{thm}
Fr\'echet space $\Schwartz$, together with Mellin convolution, is a commutative Fr\'echet algebra.
\end{thm}
\begin{proof}
It is trivial to see that for every $c\in\reals$ and $f,g \in \Schwartz$ we have
$$c \cdot(f\star g) = (c\cdot f)\star g = f \star (c\cdot g).$$

\noindent
Furthermore, distributivity of $\star$ is also obvious, so it remains to prove that it is associative, continuous and commutative. We accomplish this task in three steps.

\textbf{Step 1.} Let $f,g,h \in \Schwartz.$ In order to see that $\star$ is associative, it suffices to use Theorem B.3.1 in \cite{DeitmarEchterhoff}, p. 289 (or Theorem 8.8 in \cite{Rudin}, p. 164) and calculate that 
\begin{equation*}
\begin{split}
\forall_{x\in\reals_+}\ (f\star g)\star h(x) &= \int_0^{\infty}\ y^{-1}f\star g\left(xy^{-1}\right)h(y)\ dy \\
&= \int_0^{\infty}\ y^{-1} \left(\int_0^{\infty}\ z^{-1}f\left(xy^{-1}z^{-1}\right)g(z)\ dz\right) h(y)\ dy \\
&= \int_0^{\infty}\ \int_0^{\infty}\ y^{-1}z^{-1}f\left(xy^{-1}z^{-1}\right)g(z)h(y)\ dz dy \\
&\stackrel{z = uy^{-1}}{=} \int_0^{\infty}\ \int_0^{\infty}\ y^{-1}u^{-1}f\left(xu^{-1}\right)g\left(uy^{-1}\right)h(y)\ du dy \\
&= \int_0^{\infty}\ u^{-1}f\left(xu^{-1}\right)\left(\int_0^{\infty}\ y^{-1}g\left(uy^{-1}\right)h(y)\ dy\right) du \\
&= \int_0^{\infty}\ u^{-1}f\left(xu^{-1}\right)g\star h(u)\ du = f\star (g\star h)(x).
\end{split}
\end{equation*}

\textbf{Step 2.} Suppose that $d_{\Schwartz}(f_n,f) \longrightarrow 0$ and $d_{\Schwartz}(g_n,g) \longrightarrow 0.$ We prove that 
\begin{gather}
d_{\Schwartz}(f_n\star g_n,f\star g) \longrightarrow 0,
\label{convolutionconverges}
\end{gather}

\noindent
which establishes continuity of Mellin convolution. 

Fix $\alpha\in\integers, \beta \in\naturals_0$. By Lemma \ref{Schwartzfunctionintegrable} there exists a constant $M>0$ such that 
\begin{gather}
\forall_{n\in\naturals}\ \int_0^{\infty}\ y^{\alpha-(\beta+1)} |g_n(y)|\ dy < M \hspace{0.4cm}\text{and}\hspace{0.4cm} \int_0^{\infty}\ z^{\alpha-1} |f^{(\beta)}(z)|\ dz < M.
\label{verygoodconstantbyLemma1}
\end{gather}

\noindent
Consequently, we have
\begin{equation*}
\begin{split}
p_{\alpha,\beta}&(f_n\star g_n - f\star g) = \sup_{x\in\reals_+} \left|x^{\alpha}\left(\int_0^{\infty}\ y^{-1}f_n(xy^{-1})g_n(y)\ dy - \int_0^{\infty}\ y^{-1}f(xy^{-1})g(y)\ dy\right)^{(\beta)}\right| \\
&\stackrel{\text{Theorem } \ref{convolutionisCinfty}}{=} \sup_{x\in\reals_+} x^{\alpha}\left|\int_0^{\infty}\ y^{-(\beta+1)}f_n^{(\beta)}(xy^{-1})g_n(y)\ dy - \int_0^{\infty}\ y^{-(\beta+1)}f^{\beta}(xy^{-1})g(y)\ dy\right| \\
&= \sup_{x\in\reals_+} x^{\alpha}\left|\int_0^{\infty}\ y^{-(\beta+1)}(f_n^{(\beta)}(xy^{-1}) - f^{(\beta)}(xy^{-1}))g_n(y)\ dy - \int_0^{\infty}\ y^{-(\beta+1)}f^{(\beta)}(xy^{-1})(g(y) - g_n(y))\ dy\right| \\
&\leqslant \int_0^{\infty}\ y^{-(\beta+1)} |g_n(y)| \sup_{x\in\reals_+} x^{\alpha}|f_n^{(\beta)}(xy^{-1}) - f^{(\beta)}(xy^{-1})| \ dy + \sup_{x\in\reals_+}\ x^{\alpha} \int_0^{\infty}\ y^{-(\beta+1)} |f^{(\beta)}(xy^{-1})| |g(y) - g_n(y)|\ dy \\
&\leqslant p_{\alpha,\beta}(f_n-f) \int_0^{\infty}\ y^{\alpha-(\beta+1)} |g_n(y)|\ dy + \sup_{x\in\reals_+}\ x^{\alpha-\beta} \int_0^{\infty}\ z^{-(\beta+1)} |f^{(\beta)}(z^{-1})| |g(xz) - g_n(xz)|\ dz \\
&\leqslant p_{\alpha,\beta}(f_n-f) \int_0^{\infty}\ y^{\alpha-(\beta+1)} |g_n(y)|\ dy + \int_0^{\infty}\ z^{-(\beta+1)} |f^{(\beta)}(z^{-1})| \sup_{x\in\reals_+}\ x^{\alpha-\beta} |g(xz) - g_n(xz)|\ dz \\
&\leqslant p_{\alpha,\beta}(f_n-f) \int_0^{\infty}\ y^{\alpha-(\beta+1)} |g_n(y)|\ dy + p_{\alpha-\beta,0}(g-g_n)\int_0^{\infty}\ z^{-(\alpha+1)} |f^{(\beta)}(z^{-1})|\ dz \\
&= p_{\alpha,\beta}(f_n-f) \int_0^{\infty}\ y^{\alpha-(\beta+1)} |g_n(y)|\ dy + p_{\alpha-\beta,0}(g-g_n) \int_0^{\infty}\ z^{\alpha-1} |f^{(\beta)}(z)|\ dz\\
&\stackrel{\eqref{verygoodconstantbyLemma1}}{\leqslant} M(p_{\alpha,\beta}(f_n-f) + p_{\alpha-\beta,0}(g-g_n))
\end{split}
\end{equation*}

\noindent
Since $d_{\Schwartz}(f_n,f) \longrightarrow 0$ implies $p_{\alpha,\beta}(f_n-f) \longrightarrow 0$ and $d_{\Schwartz}(g_n,g) \longrightarrow 0$ implies $p_{\alpha-\beta,0}(g-g_n)\longrightarrow 0$ we conclude that $p_{\alpha,\beta}(f_n\star g_n - f\star g) \longrightarrow 0.$ Values $\alpha$ and $\beta$ were chosen arbitrarily, which finally implies \eqref{convolutionconverges}.

\textbf{Step 3.} To see that the Mellin convolution is commutative, we compute 
\begin{equation*}
\begin{split}
\forall_{x\in\reals_+}\ f\star g(x) &= \int_0^{\infty}\ y^{-1}f\left(xy^{-1}\right)g(y)\ dy \stackrel{y = zx}{=} \int_0^{\infty}\ z^{-1}f\left(z^{-1}\right)g(zx)\ dz \\
&\stackrel{z = u^{-1}}{=} \int_0^{\infty}\ u^{-1}f(u)g(xu^{-1})\ du = g\star f(x).
\end{split}
\end{equation*}
\end{proof}

Before we proceed with the study of the structure space of the Schwartz class, let us pause for a moment and deduce a counterpart for Young's convolution inequality.

\begin{thm}
Let $p,q,r \in [1,\infty)$ be such that 
$$\frac{1}{r} = \frac{1}{p} + \frac{1}{q} - 1.$$

\noindent
For every $f,g \in \Schwartz$ we have
\begin{gather}
\left(\int_0^{\infty}\ y^{-1} |f\star g(y)|^r\ dy\right)^{\frac{1}{r}} \leqslant \left(\int_0^{\infty}\ y^{-1} |f(y)|^p\ dy\right)^{\frac{1}{p}} \left(\int_0^{\infty}\ y^{-1} |g(y)|^q\ dy\right)^{\frac{1}{q}}.
\label{mellinyoung}
\end{gather}
\end{thm}
\begin{proof}
Let $F := f \circ \exp$ and observe that 
$$\|F\|_p^p = \int_{-\infty}^{\infty}\ \left|f\left(e^x\right)\right|^p\ dx \stackrel{e^x = y}{=} \int_0^{\infty}\ y^{-1} |f(y)|^p\ dy < \infty$$

\noindent
by Lemma \ref{Schwartzfunctionintegrable}. An analogous estimate holds for $G := g \circ \exp$. Next, we observe that 
\begin{equation*}
\begin{split}
\forall_{x\in\reals}\ f\star g\left(e^x\right) &= \int_0^{\infty}\ y^{-1}f\left(e^x y^{-1}\right)g(y)\ dy = \int_0^{\infty}\ y^{-1} F(x - \ln(y))g(y)\ dy \\
&\stackrel{\ln(y) = z}{=} \int_{\reals}\ F(x - z)G(z)\ dz = F\star_c G(x),
\end{split}
\end{equation*}

\noindent
where $\star_c$ is the classical convolution. By classical Young's convolution inequality (Theorem 3.9.4 in \cite{Bogachev}, p. 205-206) we have 
$$\|(f\star g)\circ \exp\|_r = \|F\star_c G\|_r \leqslant \|F\|_p \|G\|_q,$$

\noindent
which translates to
$$\left(\int_{\reals}\ \left|f\star g\left(e^x\right)\right|^r\ dx\right)^{\frac{1}{r}} \leqslant \left(\int_{\reals}\ \left|f\left(e^x\right)\right|^p\ dx\right)^{\frac{1}{p}} \left(\int_{\reals}\ \left|g\left(e^x\right)\right|^q\ dx\right)^{\frac{1}{q}}.$$

\noindent
The substitution $y = e^x$ finally proves \eqref{mellinyoung}.
\end{proof}

\section{Structure space of the Schwartz class}
\label{section:structurespaceandMellintransfrom}

Our goal in this section is to study the structure space $\Delta(\Schwartz, \star)$ of the Schwartz class $\Schwartz.$ This is a space of all nonzero, linear and continuous  functionals $m : \Schwartz \longrightarrow \reals,$ which are multiplicative:
$$\forall_{f,g\in\Schwartz}\ m(f\star g) = m(f)m(g).$$

\noindent
We endow the structure space $\Delta(\Schwartz,\star)$ with weak* topology and aim to prove that it is homeomorphic to $\reals$ (with the standard topology). This is highly reminiscent of the classical result that $\Delta(L^1(G),\star_c)$, where $G$ is a locally compact, abelian group and $\star_c$ is the classical convolution, is homeomorphic to the dual group $\widehat{G}$ (see Theorems 2.7.2 and 2.7.5 in \cite{Kaniuth}, p. 89-92 or Theorem 4.2 in \cite{Folland}, p. 88). 

The strategy we adopt is to first prove that $\Delta(\Schwartz, \star)$ is in bijective correspondence with $\reals$ (Theorem \ref{mellinbijection}) and only then establish that these spaces are in fact homeomorphic (Theorems \ref{omegaiscontinuous} and \ref{omegaisopen}). Without further ado, we set out on our journey by showing the following technical results:

\begin{lemma}
Let $x,y\in \reals_+.$ If $g\in \Schwartz$, then 
\begin{gather}
D_yg \star D_xg = g \star D_{xy}g,
\label{DxgstarDyg}
\end{gather}

\noindent
where for every $v\in \reals_+$ the operator $D_v : \Schwartz \longrightarrow \Schwartz$ is given by 
\begin{gather*}
\forall_{u\in \reals_+}\ D_vf(u) := f\left(uv^{-1}\right).
\end{gather*}
\label{powerproperty}
\end{lemma}
\begin{proof}
For every $u \in \reals_+$ we have
\begin{equation*}
\begin{split}
D_yg \star D_xg(u) &= \int_0^{\infty}\ v^{-1} D_yg\left(uv^{-1}\right) D_xg(v)\ dv \\
&= \int_0^{\infty}\ v^{-1}g\left(u(vy)^{-1}\right) g\left(vx^{-1}\right)\ dv\\
&\stackrel{v = wy^{-1}}{=} \int_0^{\infty}\ w^{-1}g\left(uw^{-1}\right) g\left(w(xy)^{-1}\right)\ dw = g\star D_{xy}g(u),
\end{split}
\end{equation*}
	
\noindent
which concludes the proof.
\end{proof}

\begin{lemma}
For every $g\in\Schwartz$ the function $x\mapsto D_xg$ is continuous. 
\label{continuityofDxg}
\end{lemma}
\begin{proof}
To prove that the function in question is continuous it suffices to prove that for every sequence $(x_n)\subset \reals_+$ convergent to $x_*\in\reals_+$ and fixed $\alpha\in\integers,\ \beta \in\naturals_0$ we have 
\begin{gather}
\lim_{n\rightarrow \infty}\ p_{\alpha,\beta}(D_{x_n}g - D_{x_*}g) = 0.
\label{limdoesDxngDxstargconverge}
\end{gather}

\noindent
Before we prove this equality in full generality, let us focus on a simpler case, when $g\in \Schwartzc.$

We have 
\begin{equation}
\begin{split}
p_{\alpha,\beta}(D_{x_n}g - D_{x_*}g) &= \sup_{u\in\reals_+}\ \left|u^{\alpha}\left(g\left(ux_n^{-1}\right) - g\left(ux_*^{-1}\right)\right)^{(\beta)}\right| \\
&= \sup_{u\in\reals_+}\ \left|u^{\alpha} \left(x_n^{-\beta} g^{(\beta)}\left(ux_n^{-1}\right) - x_*^{-\beta}g^{(\beta)}\left(ux_*^{-1}\right)\right)\right| \\
&\stackrel{u = vx_*}{=} x_*^{\alpha-\beta} \sup_{v\in\reals_+}\ \left|v^{\alpha} \left(\left(x_*x_n^{-1}\right)^{\beta} g^{(\beta)}\left(vx_*x_n^{-1}\right) - g^{(\beta)}(v)\right)\right|
\end{split}
\label{doesDxngDxstargconverge}
\end{equation}

\noindent
Fix $\eps>0$ and observe that by uniform continuity (Lemma 1.3.6 in \cite{DeitmarEchterhoff}, p. 11) of the function $v\mapsto v^{\alpha}g^{(\beta)}(v)$ there exists $\delta>0$ such that 
\begin{gather}
\forall_{\substack{v\in\reals_+\\ y\in[1-\delta,1+\delta]}}\ \left|v^{\alpha} \left(y^{\alpha} g^{(\beta)}(vy) - g^{(\beta)}(v)\right)\right| \leqslant \frac{\eps}{2x_*^{\alpha-\beta}}.
\label{unicont}
\end{gather}

\noindent
We may decrease $\delta$ (if necessary) so that we have
\begin{gather}
\forall_{y\in[1-\delta,1+\delta]}\ |y^{\beta-\alpha} - 1| \leqslant \frac{\eps}{2p_{\alpha,\beta}(g)x_*^{\alpha-\beta}}.
\label{definitionofrho}
\end{gather}

\noindent
Next, we observe the implication
$$v\not\in \supp\left(g^{(\beta)}\right) y^{-1} \ \Longrightarrow \ g^{(\beta)}(vy) = 0,$$

\noindent
which implies that for $v \not \in K_{\delta} := \supp\left(g^{(\beta)}\right) \cdot \left[\frac{1}{1+\delta}, \frac{1}{1-\delta}\right]$ we have
$$\forall_{y\in[1-\delta,1+\delta]}\ g^{(\beta)}(vy) = 0.$$

\noindent
This means that for $y\in[1-\delta,1+\delta]$ we have
\begin{equation*}
\begin{split}
\sup_{v\in\reals_+}\ \left|v^{\alpha}\left(y^{\beta}g^{(\beta)}(vy) - g^{(\beta)}(v)\right)\right| &= \sup_{v\in K_{\rho}}\ \left|v^{\alpha}\left(y^{\beta}g^{(\beta)}(vy) - g^{(\beta)}(v)\right)\right| \\
&\leqslant \sup_{v\in K_{\rho}}\ \left|v^{\alpha} g^{(\beta)}(vy) \left(y^{\beta} - y^{\alpha}\right)\right| + \sup_{v\in K_{\rho}}\ \left|v^{\alpha}\left(y^{\alpha}g^{(\beta)}(vy) - g^{(\beta)}(v)\right)\right| \\
&\stackrel{\eqref{unicont}}{\leqslant} p_{\alpha,\beta}(g) |y^{\beta-\alpha} - 1| + \frac{\eps}{2x_*^{\alpha-\beta}} \stackrel{\eqref{definitionofrho}}{\leqslant} \frac{\eps}{x_*^{\alpha-\beta}}.
\end{split}
\end{equation*}

\noindent
Comparison between the above estimate and \eqref{doesDxngDxstargconverge} reveals that \eqref{limdoesDxngDxstargconverge} is indeed true for $g\in \Schwartzc.$ 

It remains to prove \eqref{limdoesDxngDxstargconverge} for an arbitrary function in $\Schwartz.$ Let $g \in \Schwartz$ and $\eps>0.$ By Theorem \ref{densesubsetinSchwartz} there exists $h\in\Schwartzc$ such that $d_{\Schwartz}(g,h) < \eps.$ Observe that 
\begin{equation*}
\begin{split}
\forall_{n\in\naturals}\ p_{\alpha,\beta}(D_{x_n}g - D_{x_n}h) &= \sup_{y\in\reals_+}\ \left|y^{\alpha}\left(g(yx_n^{-1}) - h(yx_n^{-1})\right)^{(\beta)}\right| \\
&= x_n^{-\beta} \sup_{y\in\reals_+}\ \left|y^{\alpha}\left(g^{(\beta)}(yx_n^{-1}) - h^{(\beta)}(yx_n^{-1})\right)\right| \\
&\stackrel{y = zx_n}{=} x_n^{\alpha-\beta} \sup_{z\in\reals_+}\ \left|z^{\alpha}\left(g^{(\beta)}(z) - h^{(\beta)}(z)\right)\right| = x_n^{\alpha-\beta} p_{\alpha,\beta}(g-h) 
\end{split}
\end{equation*}

\noindent
and, analogously, $p_{\alpha,\beta}(D_{x_*}g - D_{x_*}h) = x_*^{\alpha-\beta} p_{\alpha,\beta}(g-h).$ Finally, we have 
\begin{equation*}
\begin{split}
\forall_{n\in\naturals}\ p_{\alpha,\beta}(D_{x_n}g - D_{x_*}g) &\leqslant p_{\alpha,\beta}(D_{x_n}g - D_{x_n}h) + p_{\alpha,\beta}(D_{x_n}h - D_{x_*}h) + p_{\alpha,\beta}(D_{x_*}h - D_{x_*}g) \\
&\leqslant x_n^{\alpha-\beta}p_{\alpha,\beta}(g-h) + p_{\alpha,\beta}(D_{x_n}h - D_{x_*}h) + x_*^{\alpha-\beta}p_{\alpha,\beta}(g-h),
\end{split}
\end{equation*}

\noindent
so passing to the limit as $n\longrightarrow \infty,$ we obtain
$$\lim_{n\rightarrow\infty}\ p_{\alpha,\beta}(D_{x_n}g - D_{x_*}g) \leqslant 2x_*^{\alpha-\beta}p_{\alpha,\beta}(g-h) + \lim_{n\rightarrow\infty}\ p_{\alpha,\beta}(D_{x_n}h - D_{x_*}h) = 2x_*^{\alpha-\beta}p_{\alpha,\beta}(g-h) < 2x_*^{\alpha-\beta}\eps.$$

\noindent
Since $\eps$ was chosen arbitrarily, we conclude the proof.
\end{proof}

We are now ready to prove that there exists a bijective correspondence between $\Delta(\Schwartz,\star)$ and $\reals$. 

\begin{thm}
If $m_s : \Schwartz \longrightarrow \reals$ is given by
\begin{gather}
m_s(f) := \int_0^{\infty}\ x^{s-1}f(x)\ dx
\label{formulaformphi}
\end{gather}

\noindent
for some $s\in\reals,$ then $m_s \in \Delta(\Schwartz,\star).$ Furthermore, for every $m \in \Delta(\Schwartz,\star)$ there exists a unique $s \in \reals$ such that $m = m_s.$
\label{mellinbijection}
\end{thm}
\begin{proof}
Obviously, if $m_s$ is given by formula (\ref{formulaformphi}) then it is a linear functional. To see that it is continuous we choose a sequence $(f_n)\subset \Schwartz,$ which is convergent to $f\in\Schwartz$ and prove that 
$$\lim_{n\rightarrow\infty}\ m_s(f_n) = m_s(f)$$ 

\noindent
The convergence of the sequence $(f_n)$ means in particular that for every $\eps>0$ there exists $N_{\eps} \in\naturals$ such that for every $n\geqslant N_{\eps}$ we have 

\begin{gather}
\sup_{x\in (0,1)}\ x^{\lfloor s-1 \rfloor} |f_n(x) - f(x)| \leqslant \eps
\label{floorinequality}
\end{gather}

\noindent
and
\begin{gather}
\sup_{x\in [1,\infty)}\ x^{\lceil s-1 \rceil + 2} |f_n(x) - f(x)| \leqslant \eps.
\label{ceilinginequality}
\end{gather}

\noindent
Consequently, for every $n\geqslant N_{\eps}$ we have 
$$\int_0^1\ x^{s-1} |f_n(x) - f(x)|\ dx \leqslant \eps\cdot \int_0^1\ \frac{x^{s-1}}{x^{\lfloor s-1\rfloor}}\ dx \leqslant \eps$$

\noindent
and
$$\int_1^{\infty}\ x^{s-1} |f_n(x) - f(x)|\ dx \leqslant \eps\cdot \int_1^{\infty}\ \frac{x^{s-1}}{x^{\lceil s-1\rceil}}\cdot x^{-2} dx \leqslant \eps.$$

\noindent
We conclude that 
$$\lim_{n\rightarrow \infty}\ \int_0^{\infty}\ x^{s-1} |f_n(x) - f(x)|\ dx = 0$$

\noindent
and thus every functional $m_s$ is continuous.

Next, in order to show that $m_s \in \Delta(\Schwartz,\star)$ we need to demonstrate that $m_s$ is muliplicative. For every $f,g\in\Schwartz$ we have 
\begin{equation*}
\begin{split}
m_s(f\star g) &= \int_0^{\infty}\ x^{s-1} f\star g(x)\ dx \\
&= \int_0^{\infty} x^{s-1} \bigg(\int_0^{\infty}\ y^{-1}f\left(xy^{-1}\right) g(y)\ dy\bigg)\ dx \\
&= \int_0^{\infty} \left(\int_0^{\infty}\ x^{s-1} f\left(xy^{-1}\right)\ dx\right) y^{-1}g(y)\ dy \\
&\stackrel{x = zy}{=} \left(\int_0^{\infty}\ z^{s-1} f(z)\ dz\right)\cdot \left(\int_0^{\infty}\ y^{s-1} g(y)\ dy\right) = m_s(f)m_s(g),
\end{split}
\end{equation*}

\noindent
where we have used Fubini's theorem (see Theorem B.3.1(b) in \cite{DeitmarEchterhoff}, p. 287 or Theorem 8.8 in \cite{Rudin}, p. 164). This concludes the first part of the proof.

For the second part of the proof, we fix an element $m\in \Delta(\Schwartz,\star)$ and try to show that there exists a unique $s \in\reals_+$ such that $m = m_s.$ We fix a nonzero function $g_*\in\Schwartz$ and note that by Lemma \ref{powerproperty} we have
$$\forall_{x,y\in \reals_+}\ D_yg_* \star D_xg_* = g_*\star D_{xy}g_*.$$

\noindent
Applying the functional $m$ to this equation and using its multiplicativity we obtain
\begin{gather}
\forall_{x,y\in \reals_+}\ m(D_yg_*) m(D_xg_*) = m(g_*) m(D_{xy}g_*).
\label{funceqnforpowers}
\end{gather}

\noindent
Next, we define the function $\phi : \reals_+ \longrightarrow \reals$ by the formula
\begin{gather}
\phi(x) := \frac{m(D_xg_*)}{m(g_*)},
\label{definitionofphi}
\end{gather}

\noindent
which is nonzero (because $\phi(1) = 1$) and continuous (by Lemma \ref{continuityofDxg}). Dividing equation (\ref{funceqnforpowers}) by $m(g_*)^2$ we may rewrite it in the form 
$$\forall_{x,y\in \reals_+}\ \phi(y) \phi(x) = \phi(xy),$$

\noindent
which is known as Cauchy's multiplicative functional equation. It follows (see Theorem 3 in \cite{Aczel}, p. 41) that $\phi(y) = y^s$ for some $s\in\reals.$ We have
\begin{equation*}
\begin{split}
\forall_{f\in\Schwartz}\ \int_0^{\infty}\ y^{s-1}f(y)\ dy &=\int_0^{\infty}\ y^{-1}f(y)\phi(y)\ dy \\
&\stackrel{\eqref{definitionofphi}}{=} \int_0^{\infty}\ y^{-1}f(y)\cdot \frac{m(D_yg_*)}{m(g_*)}\ dy\\
&= \frac{1}{m(g_*)}\cdot m\left(\int_0^{\infty}\ y^{-1}f(y)D_yg_*\ dy\right) \\
&= \frac{1}{m(g_*)}\cdot m(g_*\star f) = \frac{1}{m(g_*)}\cdot m(g_*) m(f) = m(f),
\end{split}
\end{equation*}

\noindent
where the third equality holds true due to Lemma 11.45 in \cite{AliprantisBorder}, p. 427 (or Proposition 7 in \cite{Dinculeanu}, p. 123). We have thus proved that $m = m_s.$ 

Finally, the fact that \eqref{definitionofphi} is a unique $\phi$ such that $m=m_{\phi}$ follows from a technique, which is well-known in the field of variational calculus: suppose that there exist two nonzero, continuous functions $\phi_1,\phi_2$ such that $m = m_{\phi_1} = m_{\phi_2}.$ Consequently, we have 
\begin{gather}
\forall_{f\in \Schwartz}\ \int_0^{\infty}\ f(x)\bigg(\phi_1(x) - \phi_2(x)\bigg)\ dx = 0.
\label{duBoiszero}
\end{gather}

\noindent
Assuming the existence of an element $x_*\in \reals_+$ such that $\phi_1(x_*)\neq \phi_2(x_*)$ there exists $\eps >0$ such that $\phi_1 - \phi_2$ is of (without loss of generality) positive sign on $[x_*-\eps, x_*+\eps].$ Let $f_* \in \Schwartzc$ be a function defined by the formula
\begin{gather*}
f_*(x) := \left\{\begin{array}{cc}
(\phi_1(x) - \phi_2(x))\cdot \theta\left(\frac{2}{\eps} (x-x_*+\eps)\right) & x\in \left[x_* - \eps, x_* - \frac{\eps}{2}\right]\\
\phi_1(x) - \phi_2(x) & x\in \left[x_* - \frac{\eps}{2}, x_* + \frac{\eps}{2}\right]\\
(\phi_1(x) - \phi_2(x))\cdot \theta\left(2 - \frac{2}{\eps} (x-x_*)\right) & x\in \left[x_* + \frac{\eps}{2}, x_* + \eps\right]\\
0 & x \not\in [x_*-\eps,x_*+\eps]
\end{array}\right.
\end{gather*}

\noindent
where $\theta$ is the Fabius function. It remains to observe that 
$$0 \stackrel{\eqref{duBoiszero}}{=} \int_0^{\infty}\ f_*(x)\bigg(\phi_1(x) - \phi_2(x)\bigg)\ dx \geqslant \int_{x_*-\frac{\eps}{2}}^{x_*+\frac{\eps}{2}}\ \bigg(\phi_1(x) - \phi_2(x)\bigg)^2\ dx > 0,$$

\noindent
which is a contradiction. This means that $\phi_1 = \phi_2$ and concludes the proof.
\end{proof}

In summary, Theorem \ref{mellinbijection} establishes that the function $\omega: \Delta(\Schwartz,\star) \longrightarrow\reals$ given by $\omega(m) := s$ (where $s\in\reals$ is a unique element such that $m=m_s$) is a bijection. Our next task is to prove that $\omega$ is in fact more than just a bijection of two sets. We intend to show that $\omega$ is a continuous map and ultimately a homeomorphism. We commence with the following result:

\begin{lemma}
For every constant $c\in\reals$ the function $E_c:\reals \longrightarrow \reals,$ defined by
$$E_c(s) := \int_1^3\ (x^s - x^c) \exp\left(-\frac{1}{(3-x)(x-1)}\right)\ dx,$$

\noindent
is increasing.
\label{Ecisincreasing}
\end{lemma}
\begin{proof}
By Theorem 3.2 in \cite{Lang}, p. 337-339 we have 
\begin{equation*}
\begin{split}
\frac{\partial}{\partial s}E_c(s) &= \frac{\partial}{\partial s} \int_1^3\ (x^s - x^c) \exp\left(-\frac{1}{(3-x)(x-1)}\right)\ dx \\
&= \int_1^3\ x^s \ln(x) \exp\left(-\frac{1}{(3-x)(x-1)}\right)\ dx > 0,
\end{split}
\end{equation*}

\noindent
which implies that $E_c$ is increasing and concludes the proof.
\end{proof}

We are ready to prove continuity of $\omega:$

\begin{thm}
Function $\omega : (\Delta(\Schwartz, \star),\tau^*) \longrightarrow \reals$ is continuous.
\label{omegaiscontinuous}
\end{thm}
\begin{proof}
We fix $s_*\in\reals$ as well as its arbitrary open neighbourhood $U_{\eps} := (s_*-\eps,s_*+\eps)$ where $\eps>0.$ Our task is to prove that 
$$\omega^{-1}(U_{\eps}) = \bigg\{m \in \Delta(\Schwartz,\star)\ :\ |\omega(m) - s_*| < \eps\bigg\}$$

\noindent
is weak* open and we do it by fixing an arbitrary element $m_{**}\in \omega^{-1}(U_{\eps})$ and constructing its weak* open neighbourhood $W_{**}$ contained in $\omega^{-1}(U_{\eps}),$ i.e., 
$$m_{**}\in W_{**}\subset \omega^{-1}(U_{\eps}).$$ 

To begin with, since $m_{**}\in \omega^{-1}(U_{\eps})$ then $s_{**}:= \omega(m_{**})$ satisfies $|s_{**} - s_*| < \delta\eps$ for some $\delta \in (0,1).$ Next, by Lemma \ref{Ecisincreasing} we know that the function $E_{s_{**}}$ is increasing and since $E_{s_{**}}(s_{**}) = 0$ there exists $\rho>0$ such that 
\begin{gather}
|E_{s_{**}}(s)| < \rho \ \Longrightarrow \ |s-s_{**}| < (1-\delta)\eps.
\label{crucialpropertyofE}
\end{gather}

\noindent
We pick $f\in\Schwartzc$ defined by
\begin{gather*}
f(x) := \left\{\begin{array}{cc}
\exp\left(-\frac{1}{(3-x)(x-1)}\right) & \text{if }\ x\in[1,3],\\
0 & \text{otherwise,}
\end{array}\right.
\end{gather*}
	
\noindent
and claim that 
$$W_{**} := \bigg\{m \in \Delta(\Schwartz,\star)\ :\ |m(f) - m_{**}(f)| < \rho\bigg\}$$
	
\noindent
is the desired weak* open neighbourhood of $m_{**}.$ Indeed, for every $m\in\Delta(\Schwartz,\star)$ we have 
\begin{equation*}
m(f) - m_{**}(f) = \int_0^{\infty}\ (x^s - x^{s_{**}}) f(x)\ dx \stackrel{\supp(f) = [1,3]}{=} \int_1^3\ (x^s - x^{s_{**}}) \exp\left(-\frac{1}{(3-x)(x-1)}\right)\ dx = E_{s_{**}}(s)
\end{equation*}
	
\noindent
where $s = \omega(m).$ If $m\in W_{**}$ then $|E_{s_{**}}(s)| < \rho,$ so by \eqref{crucialpropertyofE} we have $|s-s_{**}| < (1-\delta)\eps.$ Finally, we have 
$$\forall_{m\in W_{**}}\ |s - s_*| \leqslant |s - s_{**}| + |s_{**} - s_*| < (1-\delta)\eps + \delta \eps = \eps,$$
	
\noindent
which proves that $m_{**}\in W_{**} \subset \omega^{-1}(U_{\eps}).$
\end{proof}

We are on the right track to prove that $\omega$ is a homeomorphism between $\Delta(\Schwartz)$ and $\reals.$ We already know that it is a conitnuous map so it suffices to prove that it is also open. To this end we will apply the following result:

\begin{lemma}
Let $(f_k)_{k=1}^n \subset \Schwartz$ be a finite sequence of functions and let $s_* \in \reals.$ For every $\eps>0$ there exist an open neighbourhood $W_*$ of $s_*$ and constants $a,b \in \reals_+$ such that $a<1<b$ and 
$$\forall_{\substack{s\in W_*\\ k=1,\ldots,n}}\ \int_{\reals_+\backslash [a,b]}\ |f_n(x)| |x^s - x^{s_*}|\ dx < \eps.$$
\label{existenceofWsstar}
\end{lemma}
\begin{proof}
We divide the proof with respect to the value $s_*:$

\begin{itemize}
	\item Suppose that $s_* > 0$ and set $W_{s_*} := \left(\frac{1}{2}s_*,\frac{3}{2}s_*\right),$ which implies
	\begin{gather}
	\forall_{\substack{s\in W_{s_*}\\ x\in \reals_+}}\ x^s \leqslant x^{\frac{1}{2}s_*} + x^{\frac{3}{2}s_*}.
	\label{estimateonxs}
	\end{gather}
	
	Furthermore, since $(f_k)_{k=1}^n \subset \Schwartz$ then there exist constants $M_1,M_2 > 0$ such that 
	\begin{gather}
	\forall_{\substack{x\in(0,1)\\ k=1,\ldots,n}}\ |f_k(x)|\cdot \left(x^{\frac{1}{2}s_* - 1} + x^{s_* - 1} + x^{\frac{3}{2}s_* - 1}\right) < M_1
	\label{oszacowaniezM1}
	\end{gather}
	
	\noindent
	and 
	\begin{gather}
	\forall_{\substack{x\in(1,\infty)\\ k=1,\ldots,n}}\ |f_k(x)|\cdot \left(x^{\frac{1}{2}s_* + 2} + + x^{s_* + 2} + x^{\frac{3}{2}s_* + 2}\right) < M_2.
	\label{oszacowaniezM2}
	\end{gather}
	
	Finally, if $a,b$ are chosen such that $a<1<b$ and 
	$$\frac{M_1a^2}{2} + \frac{M_2}{b} < \eps$$
	
	\noindent
	then we have 
	\begin{gather*}
	\forall_{\substack{s\in W_{s_*}\\ k=1,\ldots,n}}\ \int_{\reals_+\backslash [a,b]}\ |f_k(x)| |x^s - x^{s_*}|\ dx \stackrel{\eqref{estimateonxs}}{\leqslant} \int_{\reals_+\backslash [a,b]}\ |f_k(x)|\cdot \left(x^{\frac{1}{2}s_*} + x^{s_*} + x^{\frac{3}{2}s_*}\right)\ dx \\
	\leqslant \int_0^a\ |f_k(x)|\cdot \left(x^{\frac{1}{2}s_*} + x^{s_*} + x^{\frac{3}{2}s_*}\right)\ dx + \int_b^{\infty}\ |f_k(x)|\cdot \left(x^{\frac{1}{2}s_*} + x^{s_*} + x^{\frac{3}{2}s_*}\right)\ dx \\
	\stackrel{\eqref{oszacowaniezM1}, \eqref{oszacowaniezM2}}{\leqslant} \int_0^a\ M_1x\ dx + \int_b^{\infty}\ \frac{M_2}{x^2}\ dx = \frac{M_1a^2}{2} + \frac{M_2}{b} < \eps.
	\end{gather*}
	
	\item If $s_* = 0$ set $W_{s_*} := (-1,1),$ which implies
	\begin{gather}
	\forall_{\substack{s\in W_{s_*}\\ x\in \reals_+}}\ x^s \leqslant x^{-1} + x.
	\label{estimateonxscases0}
	\end{gather}
	
	Furthermore, since $(f_k)_{k=1}^n \subset \Schwartz$ then there exist constants $M_1,M_2 > 0$ such that 
	\begin{gather}
	\forall_{\substack{x\in(0,1)\\ k=1,\ldots,n}}\ |f_k(x)|\cdot \left(x^{-2} + x^{-1} + 1\right) < M_1
	\label{oszacowaniezM1cases0}
	\end{gather}
	
	\noindent
	and 
	\begin{gather}
	\forall_{\substack{x\in(1,\infty)\\ k=1,\ldots,n}}\ |f_k(x)|\cdot \left(x + x^2 + x^3\right) < M_2.
	\label{oszacowaniezM2cases0}
	\end{gather}
	
	Finally, if $a,b$ are chosen such that $a<1<b$ and 
	$$\frac{M_1a^2}{2} + \frac{M_2}{b} < \eps$$
	
	\noindent
	then we have 
	\begin{gather*}
	\forall_{\substack{s\in W_{s_*}\\ k=1,\ldots,n}}\ \int_{\reals_+\backslash [a,b]}\ |f_k(x)| |x^s - 1|\ dx \stackrel{\eqref{estimateonxscases0}}{\leqslant} \int_{\reals_+\backslash [a,b]}\ |f_k(x)|\cdot \left(x^{-1} + 1 + x\right)\ dx \\
	\leqslant \int_0^a\ |f_k(x)|\cdot \left(x^{-1} + 1 + x\right)\ dx + \int_b^{\infty}\ |f_k(x)|\cdot \left(x^{-1} + 1 + x\right)\ dx \\
	\stackrel{\eqref{oszacowaniezM1cases0}, \eqref{oszacowaniezM2cases0}}{\leqslant} \int_0^a\ M_1x\ dx + \int_b^{\infty}\ \frac{M_2}{x^2}\ dx = \frac{M_1a^2}{2} + \frac{M_2}{b} < \eps.
	\end{gather*}
	
	\item The case when $s_* < 0$ is completely analogous to the one where $s_* > 0.$
\end{itemize}
\end{proof}

Finally, we are ready to prove openness of $\omega$:

\begin{thm}
$\omega: (\Delta(\Schwartz,\star), \tau^*)\longrightarrow \reals$ is an open map. 
\label{omegaisopen}
\end{thm}
\begin{proof}
Our task is to prove that the image (under $\omega$) of an arbitrary weak* open set 
\begin{gather*}
U := \bigg\{m\in \Delta(\Schwartz,\star) \ :\ \forall_{k=1,\ldots,n}\ |m(f_k) - m_{s_*}(f_k)| < \eps\bigg\},
\label{weakstaropenset}
\end{gather*}

\noindent
where $\eps>0,\ s_* \in \reals$ and $(f_k)_{k=1}^n\subset \Schwartz$ are fixed, is open in $\reals$. We fix $s_{**}\in \omega(U),$ which means that there exists $\delta\in(0,1)$ such that 
\begin{gather}
\forall_{k=1,\ldots,n}\ |m_{s_{**}}(f_k) - m_{s_*}(f_k)| < \delta\eps.
\label{deltaeps}
\end{gather}

\noindent
Furthermore, by Lemma \ref{existenceofWsstar} we pick $W_{s_{**}}$ and $a,b\in\reals_+$ such that $a<1<b$ and
\begin{gather}
\forall_{\substack{s \in W_{s_{**}}\\ k=1,\ldots,n}}\ \int_{\reals_+ \backslash [a,b]}\ |f_k(x)||x^s - x^{s_{**}}|\ dx \leqslant \frac{(1-\delta)\eps}{4}.
\label{choiceofK}
\end{gather}

\noindent
Last but not least, we put
\begin{gather}
V := \bigg\{s \in \reals\ :\ \sup_{x\in [a,b]}\ \left|x^s - x^{s_{**}}\right| < \frac{(1-\delta)\eps}{2\max_{k=1,\ldots,n}\ \|f_k\|_1}\bigg\} \cap W_{s_{**}},
\label{Vopenintucc}
\end{gather}

\noindent
which is an open neighbourhood of $s_{**}.$ Finally, we calculate that for every $s \in V$ and $k=1,\ldots,n$ we have
\begin{equation*}
\begin{split}
|m_s(f_k) - m_{s_*}(f_k)| &\leqslant |m_s(f_k) - m_{s_{**}}(f_k)| + |m_{s_{**}}(f_k) - m_{s_*}(f_k)|\\
&\stackrel{\eqref{deltaeps}}{\leqslant} |m_s(f_k) - m_{s_{**}}(f_k)| + \delta\eps \\
&\leqslant \int_a^b\ |f_k(x)| \left|x^s - x^{s_{**}}\right|\ dx + \int_{\reals_+\backslash [a,b]}\ |f_k(x)| \left|x^s - x^{s_{**}}\right|\ dx + \delta\eps\\
&\stackrel{\eqref{choiceofK}}{\leqslant} \int_a^b\ |f_k(x)| |x^s - x^{s_{**}}|\ dx + \frac{(1-\delta)\eps}{2} + \delta\eps\\
&\stackrel{\eqref{Vopenintucc}}{<} \frac{(1-\delta)\eps}{2\max_{k=1,\ldots,n}\ \|f_k\|_1} \int_{\reals_+}\ |f_k(x)|\ dx + \frac{(1+\delta)\eps}{2} \leqslant \eps.
\end{split}
\end{equation*}

\noindent
We conclude that $V$ is a open neighbourhood of (an arbitrarily chosen) $s_{**}$ and $V \subset \omega(U).$ Thus $\omega$ is an open map.
\end{proof}

\section{Final remarks on Mellin transform}
\label{section:finalremarks}

Entering the final section of the paper, we aim to prove that the Mellin transform is constructed in a similar fashion to the classical Gelfand transform (see Chapter 2.4 in \cite{DeitmarEchterhoff} or Chapter 2 in \cite{Kaniuth}). The classical Gelfand theory states that if $\algebra$ is a commutative Banach algebra then there exists a (norm-decreasing) algebra homomorphism 
$\gamma : \algebra \longrightarrow C_0(\Delta(\algebra))$ given by the formula
$$\forall_{m\in\Delta(\algebra)}\ \gamma(f)(m) := m(f).$$

\noindent
The map $\gamma$ is referred to as the Gelfand transform. Although (as we showed earlier in Theorem \ref{Schwartzspacenotnormable}) $\Schwartz$ is not a commutative Banach algebra, we can still consider a map $\MM$ which to every $f\in\Schwartz$ assigns a function $\MM_f$ defined on the structure space $\Delta(\Schwartz)\cong \reals$ by the formula
$$\forall_{s\in\reals}\ \MM_f(s) := m_s(f) = \int_0^{\infty}\ x^{s-1}f(x)\ dx.$$ 

\noindent
The map $\MM: f\mapsto \MM_f$ is obviously the Mellin transform, whose codomain is $C(\reals)$ due to the following lemma:

\begin{lemma}
If $f\in\Schwartz$ then $\MM_f\in C(\reals).$
\end{lemma}
\begin{proof}
In order to see that $\MM_f$ is continuous, let $s_*\in\reals$ and let $(s_n)\subset \reals$ be a sequence convergent to $s_*.$ Consider the function
\begin{gather*}
g(x) := \left\{\begin{array}{cr}
x^{\lfloor s_* - 1\rfloor - 1}|f(x)| & \text{if }\ x\in(0,1],\\
x^{\lceil s_* - 1\rceil + 1}|f(x)| & \text{if }\ x\in[1,\infty),\
\end{array}\right.
\end{gather*}

\noindent
which is integrable over $\reals_+$ (since $f\in\Schwartz$) and satisfies 
$$|x^{s_n-1}f(x)| \leqslant g(x)$$

\noindent
for every $x\in\reals_+$ and almost every $n\in\naturals.$ By dominated convergence theorem it follows that 
$$\lim_{n\rightarrow\infty}\ \MM_f(s_n) = \lim_{n\rightarrow\infty}\ \int_0^{\infty}\ x^{s_n-1}f(x)\ dx = \lim_{n\rightarrow\infty}\ \int_0^{\infty}\ x^{s_*-1}f(x)\ dx = \MM_f(s_*),$$

\noindent
which establishes continuity of $\MM_f.$
\end{proof}

In the classical Gelfand theory, Riemann-Lebesgue lemma (see Lemma 3.3.2 in \cite{Deitmar}, p. 47 or Theorem 1.30 in \cite{Folland}, p. 15) states that the Gelfand transform of an element $f$ (in a commutative Banach algebra) is not only continuous, but also vanishes at infinity, i.e., $\gamma(f)\in C_0$. Unfortunately, the case of the Mellin transform is different. To see this, it suffices to take $\theta_2$ from Lemma \ref{sequencethetan} and observe that 
$$\lim_{s\rightarrow \infty}\ \MM_{\theta_2}(s) = \lim_{s\rightarrow\infty}\ \int_0^{\infty}\ x^{s-1}\theta_2(x)\ dx \geqslant \lim_{s\rightarrow\infty}\ \int_1^{2}\ x^{s-1}\ dx = \lim_{s\rightarrow\infty}\ \frac{2^s - 1}{s} = \infty.$$

Luckily, the Mellin transform is still an algebra homomorphism (just as Gelfand transform) since for every $s \in\reals$ we have
\begin{gather}
\MM_{f\star g}(s) = m_s(f\star g) = m_s(f)m_s(g) = \MM_f(s) \MM_g(s).
\label{mellinconvolutiontheorem}
\end{gather}

\noindent
In other words, we have 
$$\forall_{s\in\reals}\ \int_0^{\infty}\ x^{s-1} f\star g(x)\ dx = \left(\int_0^{\infty}\ x^{s-1}f(x)\ dx\right)\cdot \left(\int_0^{\infty}\ x^{s-1}g(x)\ dx\right),$$

\noindent
which is the Mellin counterpart of the classical convolution theorem.

\end{document}